\documentclass[letterpaper, 11 pt]{article}
\usepackage{fullpage,amsthm, amsmath, amssymb, amsfonts}

\usepackage[margin=1in]{geometry}
\usepackage{setspace}
\usepackage{graphicx}
\usepackage{mathtools}

\newtheorem{theorem}{Theorem}
\newtheorem{proposition}{Proposition}
\newtheorem{lemma}{Lemma}
\newtheorem{definition}{Definition}
\newtheorem{corollary}{Corollary}
\newtheorem{remark}{Remark}
\newcommand{\newword}[1]{\textbf{\emph{#1}}}

\newcommand{\RR}{\mathbb{R}}

\newcommand{\M}{\mathcal{M}}

\newcommand{\B}{\mathcal{B}}
\newcommand{\I}{\mathcal{I}}
\newcommand{\intersect}{\cap}
\newcommand{\union}{\cup}

\DeclarePairedDelimiter\parentheses{\lparen}{\rparen}
\newcommand{\rk}[1]{\operatorname{rk} \parentheses*{#1}}
\newcommand{\nbd}[1]{\operatorname{nbd} \parentheses*{#1}}
\newcommand{\minn}[1]{\operatorname{min} \parentheses*{#1}}
\newcommand{\maxx}[1]{\operatorname{max} \parentheses*{#1}}
\newcommand{\cw}[1]{\operatorname{cw} \parentheses*{#1}}
\newcommand{\ccw}[1]{\operatorname{ccw} \parentheses*{#1}}
\newcommand{\minelts}[1]{\operatorname{minelts} \parentheses*{#1}}

\usepackage{lipsum}

\begin{document}

\title{The rank function of a positroid and non-crossing partitions.}

\author{Robert Mcalmon and Suho Oh}
\maketitle




\begin{abstract}
A positroid is a special case of a realizable matroid, that arose from the study of totally nonnegative part of the Grassmannian by Postnikov \cite{Postnikov}. Postnikov demonstrated that positroids are in bijection with certain interesting classes of combinatorial objects, such as Grassmann necklaces and decorated permutations. The bases of a positroid can be described directly in terms of the Grassmann necklace and decorated permutation \cite{Oh}. In this paper, we show that the rank of an arbitrary set in a positroid can be computed directly from the associated decorated permutation using non-crossing partitions. 
 
\end{abstract}

\section{Introduction}

A matrix is totally positive (respectively totally nonnegative) if all its minors are positive (respectively nonnegative) real numbers. These matrices have a number of remarkable properties: for example, an $n \times n$ totally positive matrix has $n$ distinct positive eigenvalues. The space of these matrices can be grouped up into topological cells, with each cell completely parametrized by a certain planar network \cite{fomin}. The idea of total positivity found numerous applications and was studied from many different angles, including oscillations in mechanical systems, stochastic processes and approximation theory, and planar resistor networks \cite{fomin}.

Now, instead of considering $n \times n$ matrices with nonnegative minors, consider a full-rank $k \times n$ matrix with all maximal minors nonnegative. This arose from the study of the totally nonnegative part of the Grassmannian by Postnikov \cite{Postnikov}. The set of nonzero maximal minors of such matrices forms a positroid, which is a matroid used to encode the topological cells inside the nonnegative part of the Grassmannian. Positroids have a number of nice combinatorial properties. In particular, Postnikov demonstrated that positroids are in bijection with certain interesting classes of combinatorial objects, such as Grassmann necklaces and decorated permutations. Recently, positroids have seen increased applications in physics, with use in the study of scattering amplitudes \cite{arkani} and the study of shallow water waves \cite{kodama}.

The set of bases of a positroid can be described nicely from the Grassmann necklace \cite{Oh}, and the polytope coming from the bases can be described using the cyclic intervals \cite{LP},\cite{ARW}. Non-crossing partitions were used to construct positroids from its connected components in \cite{ARW}. They were also used in \cite{EP} as an analogue of the bases for electroids. In this paper, we provide yet another usage of cyclic intervals and non-crossing partition for positroids.

Given an arbitrary set, the rank (the size of the biggest intersection with a basis) can be obtained by going through all the bases. In this paper, we show a method of obtaining the rank of an arbitrary set directly from the associated decorated permutation without having to go through the bases. In particular, we get a collection of upper bounds of the rank coming from non-crossing partitions, and one of them will be shown to be tight.

The structure of the paper is as follows. In section $2$, we go over the background materials needed for this paper, including the basics of matroids, positroids, Grassmann necklaces and decorated permutations. In section $3$ we show a basis exchange like property for cyclic intervals that works for positroids. In section $4$, we show our main result: that the rank of an arbitrary set in a positroid can be obtained directly from the decorated permutation by using non-crossing partitions. In section $5$, we provide an example of how to use our main result to compute the rank of a set.

\section*{Acknowledgement}
The authors would also like to thank Lillian Bu, Wini Taylor-Williams and David Xiang for useful discussions.

\section{Background materials}

\subsection{Matroids}

In this section we review the basics of matroids that we will need. We refer the reader to \cite{Oxley} for a more in-depth introduction to matroid theory.

\begin{definition}
A \newword{matroid} is a pair $(E,\B)$ consisting of a finite set $E$, called the \newword{ground set} of the matroid, and a nonempty collection of subsets $\B = \B(\M)$ of $E$, called the \newword{bases} of $\M$, which satisfy the \newword{basis exchange axiom}:

If $B_1,B_2 \in \B$ and $b_1 \in B_1 \setminus B_2$, then there exists $b_2 \in B_2 \setminus B_1$ such that $B_1 \setminus \{b_1\} \union \{b_2\} \in \B$. 
\end{definition}

A subset $F \subseteq E$ is called \newword{independent} if it is contained in some basis. All maximal independent sets contained in a given set $A \subseteq E$ have the same size, called the \newword{rank} $\rk{A}$ of $A$. The rank of the matroid $\M$, denoted as $\rk{\M}$, is given by $\rk{E}$. An element $e \in E$ is a \newword{loop} if it is not contained in any basis. An element $e \in E$ is a \newword{coloop} if it is contained in all bases. A matroid $\M$ is \newword{loopless} if it does not contain any loops. The \newword{dual} of $\M$ is a matroid $\M^{*} = (E,\B')$ where $\B' = \{E \setminus B | B \in \B(\M)\}$. By using the basis exchange axiom on the dual matroid, we get the following \newword{dual basis exchange axiom}:

If $B_1,B_2 \in \B$ and $b_2 \in B_2 \setminus B_1$, then there exists $b_1 \in B_1 \setminus B_2$ such that $B_1 \setminus \{b_1\} \union \{b_2\} \in \B$. 

\begin{remark}
In this paper, we will always use $[n] := \{1,\ldots,n\}$ as our ground set, reserving the usage of $E$ for subsets of the ground set we analyze. A matroid of rank $d$ will have bases in the set ${[n] \choose d}$ which stands for all cardinality $d$-subsets of $[n]$.
\end{remark}

Let $E$ be an arbitrary subset of the ground set $[n]$. For a basis $J$, if $|J \intersect E|$ is maximal among $|B \intersect E|$ for all bases $B$ of the matroid $\M$, we say that $J$ \newword{maximizes} $E$, or $J$ is \newword{maximal} in $E$. Similarly, if $|J \intersect E|$ is minimal among $|B \intersect E|$ for all bases $B$ of $\M$, we say that $J$ \newword{minimizes} $E$, or $J$ is \newword{minimal} in $E$.

The following property of the rank function will be crucial:

\begin{theorem} \cite{Oxley}
The rank function is \textit{semimodular}, meaning that $\rk{A\cup B}+\rk{A\cap B}\leq \rk{A}+\rk{B}$ for any subset $A$ and $B$ of the ground set.
\end{theorem}

Consider a matrix with entries in $\RR$ that has $n$ columns and $r$ rows, with $r \leq n$. Column sets that forms a $r$-by-$r$ submatrix with nonzero determinant forms (the set of bases of) a matroid. Such matroids are called \newword{realizable} matroids. For example, consider the following matrix:
\[ A = \left( \begin{array}{cccc}
1 & 0 & -3 & -1 \\
0 & 1 & 4 & 0\\
\end{array} \right)\] 

The column sets $\{1,2\}, \{1,3\},\{2,3\},\{2,4\},\{3,4\}$ give two-by-two submatrices that has nonzero determinant. So the collection $\{\{1,2\}, \{1,3\},\{2,3\},\{2,4\},\{3,4\}\}$ is a realizable matroid.

\begin{proposition}
\label{prop:exch}
Let $\M$ be a realizable matroid over the ground set $[n]$, and let $B$ be a basis of $\M$. Pick an arbitrary subset $E$ of $[n]$ such that $B$ maximizes $E$ and some $J \subseteq E$ such that $|J|= \rk{J} = \rk{E}$. Then $B \setminus (B \intersect E) \union J$ is another basis of $\M$.
\end{proposition}
\begin{proof}
From the condition $|J| = \rk{J} = \rk{E}$, the span of the vectors indexed by the set $J$ is exactly same as the span of the vectors indexed by the set $E$. Since $B$ maximizes $E$, the span of the vectors indexed by the set $B \intersect E$ is the same vector space. Hence starting from a set of basis vectors indexed by the set $B$, if we replace the set of vectors indexed by $B \intersect E$ with the set of vectors indexed by $J$, we still get a set of basis vectors.
\end{proof}

\subsection{Positroids}

In this section we go over the basics of positroids. Positroids were originally defined in \cite{Postnikov} as the column sets coming from nonzero maximal minors in a matrix such that all maximal minors are nonnegative. For example, the matrix we saw in the previous section has nonnegative maximal minors:
\[ A = \left( \begin{array}{cccc}
1 & 0 & -3 & -1 \\
0 & 1 & 4 & 0\\
\end{array} \right)\] 

The nonzero maximal minors come from column sets $\{1,2\}, \{1,3\},\{2,3\},\{2,4\},\{3,4\}$. This collection forms a positroid. However in this paper, we will use an equivalent definition using Grassmann necklace and Gale orderings.

\begin{definition}
Let $d \leq n$ be positive integers. A \newword{Grassmann necklace} of type $(d,n)$ is a sequence $(I_1,\ldots,I_n)$ of $d$-subsets $I_k \in {[n] \choose d}$ such that for any $i \in [n]$,
\begin{itemize}
\item if $i \in I_i$ then $I_{i+1} = I_i \setminus \{i\} \union \{j\}$ for some $j \in [n]$,
\item if $i \not \in I_i$ then $I_{i+1} = I_i$,
\end{itemize}
where $I_{n+1} = I_1$. 
\end{definition}

The \newword{cyclically shifted order} $<_i$ on the set $[n]$ is the total order
$$i <_i i+1 <_i \cdots <_i n <_i 1 <_i \cdots <_i i-1. $$

For any rank $d$ matroid $\M$ with ground set $[n]$, let $I_k$ be the lexicographically minimal basis of $\M$ with respect to $<_k$, and denote
$$\I(\M) := (I_1,\ldots,I_n),$$
which forms a Grassmann necklace \cite{Postnikov}.

The \newword{Gale order} on ${[n] \choose d}$ (with respect to $<_i$) is the partial order $<_i$ defined as follows: for any two $d$-subsets $S = \{s_1 <_i \cdots <_i s_d\}$ and $T = \{t_1 <_i \cdots <_i t_d\}$ of $[n]$, we have $S \leq_i T$ if and only if $s_j \leq_i t_j$ for all $j \in [d]$ \cite{Gale}.

\begin{theorem}[\cite{Postnikov},\cite{Oh}]
\label{thm:oh}
Let $\I = (I_1,\ldots,I_n)$ be a Grassmann necklace of type $(d,n)$. Then the collection
$$\B(\I) := \{B \in {[n] \choose d} | B \geq_j I_j, \text{ for all } j \in [n] \}$$
is the collection of bases of a rank $d$ positroid $\M(\I) := ([n],\B(\I))$. Moreover, for any positroid $\M$, we have $\M(\I(\M)) = \M$. 
\end{theorem}

In order to check if a set is a basis of a positroid or not, we do not have to check for all the cyclic orderings.

\begin{corollary}
\label{cor:selected}
Let $\M \subseteq {[n] \choose d}$ be a positroid and $\I$ the associated Grassmann necklace. A set $B \in {[n] \choose d}$ is a basis of $\M$ if and only if $B \geq_b I_b$ for all $b \in B$.
\end{corollary}
\begin{proof}
For arbitrary $q \in [n]$, denote the elements of $B$ as $b_1 <_q b_2 <_q \cdots <_q b_d$. If we had $B \geq_{b_1} I_{b_1}$, we would also have $B \geq_q I_{b_1} \geq_q I_q$.
\end{proof}


\begin{definition}
A \newword{decorated permutation} of the set $[n]$ is a bijection $\pi$ of $[n]$ whose fixed points are colored either white or black. A \newword{weak $i$-exceedance} of a decorated permutation $\pi$ is an element $j \in [n]$ such that either $j <_i \pi^{-1}(j)$ or $j$ is a fixed point colored black. 
\end{definition}

Given a decorated permutation $\pi$ of $[n]$ we can construct a Grassmann necklace $\I = (I_1,\ldots,I_n)$ by letting $I_k$ be the set of weak $k$-exceedances of $\pi$. A graphical way to see this is to cut the circle off between $k-1$ and $k$ to get a horizontal straight line with leftmost endpoint being $k$ and rightmost endpoint being $k-1$. Redraw the arrows of the permutation accordingly so that it stays within the line. Endpoints of the leftward arrows are exactly the weak $k$-exceedances of $\pi$, hence the elements of $I_k$. There is a bijection between Grassmann necklaces and decorated permutations \cite{Postnikov}.

\begin{figure}
\begin{center}
\includegraphics[scale = 0.7]{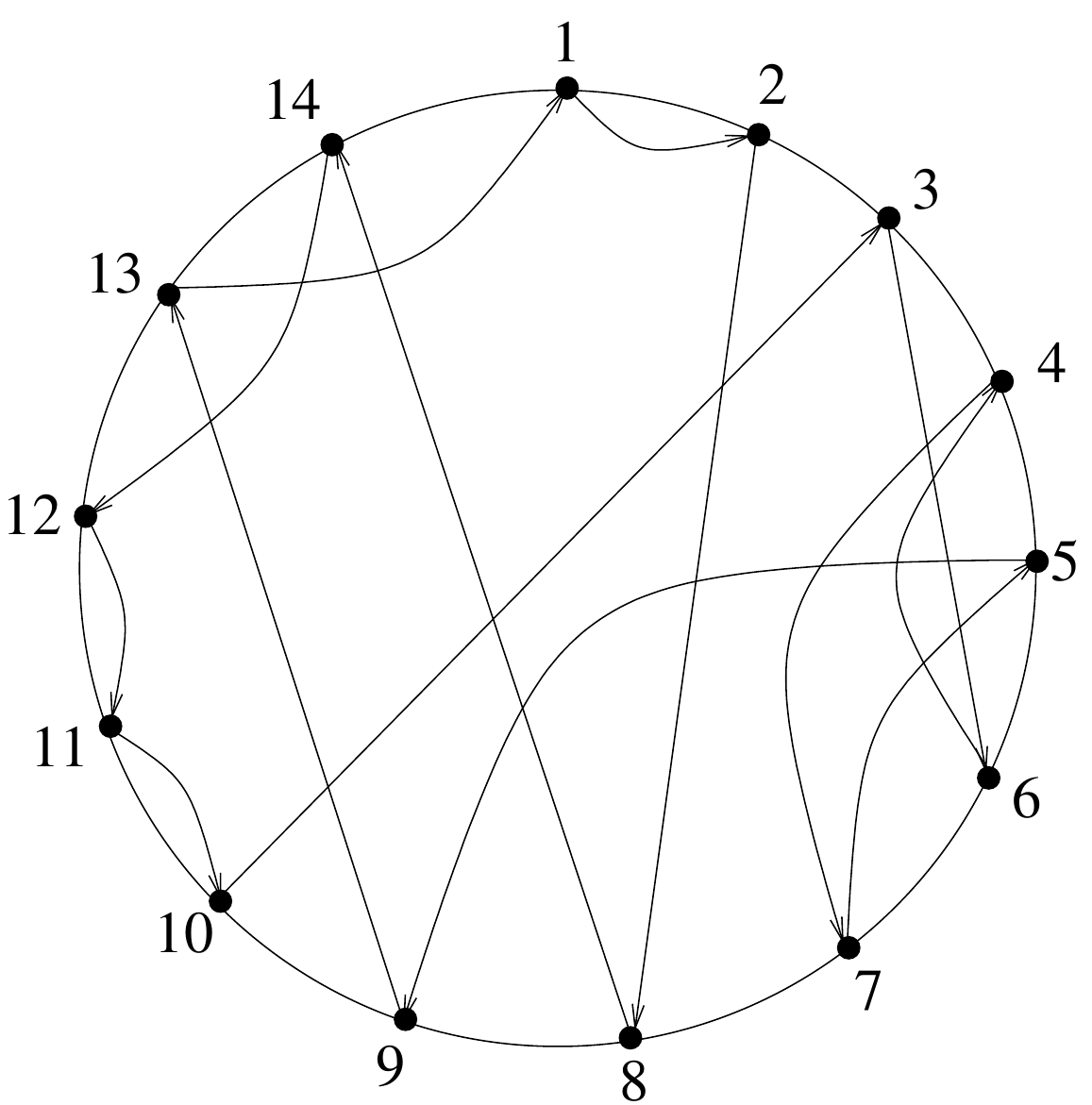}

    \caption{A decorated permutation.}
		\label{fig:permutation}

\end{center}
\end{figure}

For example, take a look at the decorated permutation (since it has no fixed points, it is the usual permutation) in Figure~\ref{fig:permutation}. It is the permutation $[2,8,6,7,9,4,5,14,13,3,10,11,1,12]$ under the usual bracket notation. The weak $1$-exceedances of the permutation is given by the set $\{1,3,4,5,10,11,12\}$, and this is $I_1$ of the associated Grassmann necklace.

\begin{remark}
When we are dealing with positroids, we will always envision the ground set $[n]$ to be drawn on a circle. We will say that $a_1,\ldots,a_t \in [n]$ are \newword{cyclically ordered} if there exists some $i \in [n]$ such that $a_1 <_i \cdots <_i a_t$.
\end{remark}

Given $a,b \in [n]$, we define the cyclic interval $[a,b]$ to be the set $\{x| x \leq_a b\}$. These cyclic intervals play an important role in the structure of a positroid \cite{Knutson},\cite{LP},\cite{ARW}. All intervals mentioned in this paper will actually be referring to cyclic intervals.

\begin{remark}
\label{rem:nofixed}
If a positroid $\M$ has loops or coloops, it is enough to study the positroid $\M'$ obtained by deleting the loops and the coloops to study the structural properties of $\M$. So throughout this paper, we will assume that our positroid has neither loops nor coloops. This means that the associated decorated permutation has no fixed points.
\end{remark}

\section{Interval exchange and Morphing}

In this section we develop a stronger basis exchange technique for positroids. Throughout the paper, unless otherwise stated, we will always be working with a positroid $\M$ on a ground set $[n]$, with rank $d$, having Grassmann necklace $\I = (I_1,\ldots,I_n)$, and an associated decorated permutation $\pi$ that does not have any fixed points (see Remark~\ref{rem:nofixed}). The example positroid that we will be using, again unless otherwise stated, will be the positroid associated to the decorated permutation of Figure~\ref{fig:permutation}.

The following property follows from the definition of Grassmann necklaces and the proof will be omitted.

\begin{lemma}[Sharing property]
\label{lem:sharing}
Let $a$ and $b$ be arbitrary elements of $[n]$. Then we have $I_a \intersect [b,a) \subseteq I_b \intersect [b,a)$.
\end{lemma}

To illustrate using our running example, notice that since $I_3=\{3,4,5,8,10,11,12\}$, the set $I_3\intersect[9,3)=\{10,11,12\}$ is contained in $I_9 = \{9,10,11,12,14,4,5\}$.

We begin our analysis of the cyclic intervals of a positroid. The following lemma follows directly from Theorem~\ref{thm:oh}.

\begin{lemma}
\label{lem:intervalrk}
For any interval $[a,b] \subseteq [n]$, the interval is maximized by $I_a$. Any interval $(b,a)$ is minimized by $I_a$.
\end{lemma}

This can easily be seen by taking some Grassmann necklace element and any arbitrary basis; say, $I_6=\{6,7,8,9,10,11,12\}$ and $B=\{6,7,10,11,12,1,4\}$.
Examine how the lemma holds on the intervals $[6,10]$ and $(10,6)$ in $[14]$:
$I_6\intersect[6,10]=\{6,7,8,9,10\}$ contains more elements than $B\intersect[6,10]=\{6,7,10\}$, while
$I_6\intersect(10,6)=\{11,12\}$ contains fewer elements than $B\intersect(10,6)=\{11,12,1,4\}$.
The above lemma also suggests that given a cyclic interval $[a,b]$, the set $I_a \intersect [a,b]$ plays a crucial role in studying that interval. The following claim follows directly from Proposition~\ref{prop:exch} and Lemma~\ref{lem:intervalrk}.

\begin{corollary}[Interval exchange property of positroids]
\label{prop:IEX}
If $J \in \M$ maximizes $[a,b] \subseteq [n]$, then $J \setminus (J \intersect [a,b]) \union (I_a \intersect [a,b]) \in \M$. Similarly, if $J$ minimizes $(b,a) \subseteq [n]$, then $J \setminus (J \intersect (b,a)) \union (I_a \intersect (b,a)) \in \M$.
\end{corollary}

Here is an example of how the interval exchange property works. In the positroid coming from Figure~\ref{fig:permutation}, we have $I_{13} = \{13,14,3,4,5,10,11\}$. The set $B = \{1,4,7,8,10,11,13\}$ is a basis of the positroid. Now if we exchange $B \intersect [13,2] = \{13,1\}$ with $I_{13} \intersect [13,2] = \{13,14\}$, the resulting set $\{4,7,8,10,11,13,14\}$ is again a basis.

Our goal of the paper is to express the rank of an arbitrary set $E \subseteq [n]$ using non-crossing partitions. To do so, we need to construct the bases that maximize $E$ and analyze them. 

\begin{remark}
\label{rem:E}
When $E$ is a subset of the ground set $[n]$ and we are trying to write $E$ as a disjoint union of cyclic intervals so that $E = [a_1,b_1] \union \cdots \union [a_s,b_s]$, we will arrange the $a_i$'s such that $a_1 < a_2 < \cdots < a_s$ unless otherwise stated. The symbol $s$ will always be reserved for the number of disjoint intervals that $E$ has. Here the indices of $[s]$ are considered cyclically, so $a_{s+1}=a_1$.
\end{remark}

Our goal is to show that it is possible to find a basis that maximizes $E$ starting from some Grassmann necklace element and then applying a series of transformations to it.

\begin{lemma}
\label{lem:transf}
Let $E$ be an arbitrary subset of $[n]$. Write $E$ as in Remark~\ref{rem:E}. Let $i$ be any element of $[s]$. Then there exists a basis $B$ that maximizes $E$ and satisfies $B \intersect (b_{i-1},b_{i}] = I_{a_i} \intersect (b_{i-1},b_i]$. 
\end{lemma}
\begin{proof}
Let $B$ be a basis which maximizes $E$.
Pick any $e \in B\intersect(b_{i-1},a_i)\setminus I_{a_i}$. By the basis exchange axiom, there is an $e' \in I_{a_i}\setminus B$ such that $(B\setminus\{e\})\union\{e'\}$ is a basis; furthermore, this maximizes $E$. Set this as new $B$, and repeat the process until we run of elements in $B\intersect(b_{i-1},a_i)\setminus I_{a_i}$.

Now, let $e' \in I_{a_i}\intersect[a_i,b_i]\setminus B$. By the dual basis exchange axiom, there is an $e \in B\setminus I_{a_i}$ such that $(B\setminus\{e\})\union\{e'\}$ is a basis; furthermore, this maximizes $E$. Set this as new $B$, and repeat the process until we run of elements in $I_{a_i}\intersect[a_i,b_i]\setminus B$. 


\end{proof}

In particular, $B$ as above will minimize $(b_{i-1},a_i)$ and maximize $[a_i,b_i]$. To illustrate the above lemma with our running example, let $E=[1,4]\union[6,7]$. The set $B^0=\{1,3,6,7,10,11,14\}$ happens to be a basis which maximizes $E$.
Recall that $I_1=\{1,3,4,5,10,11,12\}$.
By exchanging to get $B^1:=(B^0\setminus\{14\})\union\{12\}$, we have another basis which maximizes $E$ and satisfies the condition that $B^1\intersect(7,1)=I_1\intersect(7,1)$. By exchanging again to get $B^2=(B^1\setminus\{6\})\union\{4\}$, we arrive at a final basis $B^2$ satisfying the condition that $B^2\intersect[1,4]=I_1\intersect[1,4]$ as well.

Now we develop a method of constructing a basis that maximizes $E$, starting from some element of the Grassmann necklace. Fix some cyclically ordered elements $b,c,d \in [n]$. Recall that the number of elements a basis can have in the interval $(b,c)$ is bounded below by $|I_c \intersect (b,c)|$ and the number of elements a basis can have in the interval $[c,d]$ is bounded above by $|I_c \intersect [c,d]|$. We will say that a set $J$ is \newword{compatible} with $I_c$ in $(b,d]$ if $J \intersect (b,c) \supseteq I_c \intersect (b,c)$ and $J \intersect [c,d] \subseteq I_c \intersect [c,d]$. When we are comparing such $J$ with $I_c$, we will call the elements of $(J \setminus I_c) \intersect (b,c)$ as the \newword{excessive} elements and the elements of $(I_c \setminus J) \intersect [c,d]$ as the \newword{gaps}. 


A set $J$ \newword{mimics} $I_c$ in $(b,d]$ if it is compatible with $I_c$ in $(b,d]$ and at least one of the above containments is an equality. For a set $J$ that mimics $I_c$ in $(b,d]$, if we have $J \intersect [c,d] = I_c \intersect [c,d]$, we say that $J$ is \newword{gap-free} (with respect to $I_c$ in $(b,d]$). Otherwise we say that $J$ \newword{has gaps} (with respect to $I_c$ in $(b,d]$).

\begin{remark}
For any cyclically ordered $a,b,c,d$, we have that $I_a$ is compatible with $I_c$ in $(b,d]$ from the sharing property, but it doesn't necessarily mimic $I_c$ in the same interval.
\end{remark}

Let $J$ be a basis of $\M$ that is compatible to $I_c$ in $(b,d]$, where $b,c,d$ are cyclically ordered elements of $[n]$. Our goal is to transform $J$ into a basis that mimics $I_c$ in $(b,d]$. The idea is to replace the elements of $(J \setminus I_c) \intersect (b,c)$ with $(I_c \setminus J) \intersect [c,d]$. Let $\alpha$ be $\minn{|(J \setminus I_c) \intersect (b,c)|,|(I_c \setminus J) \intersect [c,d]|}$. Define $J'$ to be the set obtained from $J$ by replacing biggest (with respect to $<_b$) $\alpha$ elements of $(J \setminus I_c) \intersect (b,c)$ with the smallest (again with respect to $<_b$) $\alpha$ elements of $(I_c \setminus J) \intersect [c,d]$. We will say that $J'$ is obtained from $J$ by \newword{mimicking $I_c$ in $(b,d]$}. We will describe the process as excessive elements of $(J \setminus I_c) \intersect (b,c)$ being moved to fill the gaps of $(I_c \setminus J) \intersect [c,d]$. The newly created $J'$ mimics $J$. We say that this mimicking process has gaps or is gap-free depending on whether $J'$ has gaps or is gap-free (with respect to $I_c$ in $(b,d]$).

Now we will use the above process multiple times starting from a Grassmann necklace element and produce multiple sets, that will potentially be a basis that maximizes $E$ (again using Remark~\ref{rem:E}). We dedicate $J^0$ to stand for $I_{a_1}$. Recursively, $J^{t}$ is going to be obtained from $J^{t-1}$ by mimicking $I_{a_{t+1}}$ in $(b_t,b_s]$ for $t \in \{1,\ldots,s-1\}$ (this is possible since $J^{t-1}$ is compatible with $I_{a_{t+1}}$ in $(b_t,b_s]$). We call this process the $t$-th \newword{morph} of $J^0 = I_{a_1}$. So we will say that the set $J^t$ is obtained from $J^0$ by \newword{morphing} $t$ times. Similarly, we will use $J_i^t$ to denote the set obtained from $I_{a_i}$ by imposing $a_1 <_i \cdots <_i a_s$ when labeling the starting points of the intervals of $E$, then morphing $t$ times.


For example, consider the set $E=[2,4]\union[7,10]$ in our example from Figure~\ref{fig:permutation}. The set $J_1^0$ is defined as $I_2 = \{2,3,4,5,10,11,12\}$. Since we are dealing with $J_1$, we label $a_1 = 2, b_1 = 4, a_2 = 7, b_2 = 10$. The first morph of $J_1$ will be mimicking $I_7$ in $(4,10]$. From $I_7 = \{7,8,9,10,11,12,4\}$, we move the excessive elements $(J_1^0 \setminus I_7) \intersect (4,7) = \{5\}$ to fill the gaps of $(I_7 \setminus J_1^0) \intersect [7,10] = \{7,8,9\}$. This gives us $J_1^1 = \{2,3,4,7,10,11,12\}$, which has gaps (with respect to $I_7$ in $[7,10]$). The reader should beaware that we do not know if $J_1^1$ is actually a basis of $\M$ yet. Similarly, the set $J_2^0$ is defined as $I_7 = \{7,8,9,10,11,12,4\}$. When we are dealing with $J_2$, we label $a_1 = 7, b_1 = 10, a_2 = 2, b_2=4$. The first morph of $J_2$ will be mimicking $I_2$ in $(10,4]$. Since there are no excessive elements in $(10,2)$, we have $J_2^1 = J_2^0$ in this case.


We have an analogue of the sharing property for $J_i^t$'s, which is straightforward from the sharing property:

\begin{lemma}[Sharing property for the morphs]
\label{lem:sharingmorph}
We have $J_i^t \cap [a_{i+1},a_i) \subseteq J_{i+1}^{t-1} \cap [a_{i+1},a_i)$.
\end{lemma}

Our ultimate goal is to show that one of the $J_i^t$'s will maximize $E$.

\begin{lemma}
\label{lem:earlymax}
Fix a subset $E$ of the ground set as in Remark~\ref{rem:E}. Fix some $1 \leq t \leq s-1$, then consider the set $J^t$. For each $1 \leq p \leq s$, there exists some nonnegative number $q$ and a sequence $i_1 < \cdots < i_q < i_{q+1}=p$ such that $J^t$ maximizes $[a_1,b_{i_1}], [a_{i_1+1},b_{i_2}], \ldots, [a_{i_q+1},b_p]$.


\end{lemma}
\begin{proof}
Recall that the $t$-th morph removes the excessive elements in $(b_t,a_{t+1})$ and fills the gaps of $[a_{t+1},b_s]$ from left to right. Consider the intervals $[a_1,b_1], [a_2,b_2], \ldots, [a_p,b_p]$ which are some of the components of $E$. We will associate a number on each interval in the following way : for each $[a_x,b_x]$, let $\gamma(x)$ denote the biggest number within $\{0,\ldots,\maxx{t,x-1}\}$ such that all gaps of $[a_{\gamma(x)+1},b_x]$ gets filled in the $\gamma(x)$-th morph (that is, when $J^{\gamma(x)} \intersect [a_{\gamma(x)+1},b_x] = I_{a_{\gamma(x)+1}} \intersect [a_{\gamma(x)+1},b_x]$). Such number is guaranteed to exist, since $J^0 = I_{a_1}$. Now $J^t$ maximizes $[a_{\gamma(x)+1},b_x]$ in $\M$, since the morphs after the $\gamma(x)$-th morph does not change the number of elements in that interval. Starting from $p$, take $\gamma(p),\gamma(\gamma(p)),\ldots$ until you get $0$. Delete $0$ from this collection, and relabel them as $i_1 < i_2 < \cdots$ to get the desired result.
\end{proof}


From the above lemma, we are guaranteed that each $J^t$ maximizes some set in $[a_1,b_t]$ (setting $p$ as $t$) which is obtained from $E \intersect [a_1,b_t]$ by merging some nearby intervals and replacing them with a bigger interval (for example merging $[1,3] \union [6,9]$ to get $[1,9]$). Now if $J^t$ was gap free (that is the morph to get $J^t$ from $J^{t-1}$ is gap-free) then $[a_{t+1},b_s]$ is also maximized. In other words, $J^t$ that is gap free will maximize some set that is obtained from $E$ by merging some nearby intervals. 

The remainder of this section will be dedicated to showing that there is some $i$ and $t$ such that $J_i^t$ is gap free and is a basis of the positroid. The next section will use that result to obtain our main result. 


\begin{lemma}
\label{lem:push}
Let $b,c,d,e$ be cyclically ordered elements of the ground set $[n]$ and let $J$ be a basis of $\M$ that is compatible to $I_c$ in $(b,d]$. Define $J'$ to be obtained from $J$ by mimicking $I_c$ in $(b,d]$. The following holds:
\begin{itemize}
 \item If $J' \in \M$ and $J'$ has gaps (with respect to $I_c$ in $(b,d]$), then we have $|J \intersect (z,b]| \geq |I_c \intersect (z,b]|$ for all $z \in (d,b]$.

 \item If $|J \intersect (z,b]| \geq |I_c \intersect (z,b]|$ for all $z \in (d,b]$, then $J' \in \M$. Moreover if we have $J' \intersect [c,e) \subseteq I_c \intersect [c,e)$, then the above inequality holding for all $z \in [e,b]$ is enough to get $J' \in \M$.

\end{itemize}
\end{lemma}
\begin{proof}

From the fact that $J' \setminus J >_z J \setminus J'$ for any $z \not \in (b,d]$ and using Corollary~\ref{cor:selected}, it is enough to show $J' \geq_c I_c$ in order to achieve $J' \in \M$. Therefore $J' \in \M$ is equivalent to $|J' \intersect [c,z]| \leq |I_c \intersect [c,z]|$ for all $z \in [n]$.  Since $J$ and $J'$ are compatible with $I_c$ in $(b,d]$ the inequality automatically holds for any $z \in (b,d]$. Hence we only need to show $|J' \intersect (z,c)| \geq |I_c \intersect (z,c)|$ for all $z \in (d,b]$. Observe that $|J' \intersect (z,c]| = |J \intersect (z,c]| - \alpha$ for $z \in (d,b]$, where $\alpha$ is the minimum of $|(I_c \setminus J) \intersect [c,d]|$ (process is gap-free) and $|(J \setminus I_c) \intersect (b,c)|$ (has gaps). Cleaning up the inequalities in the latter case gives us the desired results. 
\end{proof}

Using the above lemma, we will finish off the section with the following result.

\begin{proposition}
\label{prop:push}

Let $\M$ be a positroid over the ground set $[n]$ and $E = [a_1,b_1] \union \cdots [a_s,b_s]$ be a subset of the ground set. Again consider the sets of form $J_i^t$ for $1 \leq t \leq s-1$, obtained from $I_{a_i}$ by morphing $t$ times with respect to $E$. Fix some $1 \leq h \leq s-1$. If $J_i^t \in \M$ and have gaps for all $i \in [n]$ and $1 \leq t < h$, then $J_i^h \in \M$ for all $i \in [n]$.
\end{proposition}

\begin{proof}

First of all, $J^1 \in \M$ follows from Lemma~\ref{lem:push} and the sharing property. Hence we only need to consider the case when $h>1$. We will show that if $J^1,\ldots,J^{h-1},J_2^1,\ldots,J_2^{h-1}$ are bases of $\M$ and have gaps, then $J^h \in \M$. Also notice that $J^h \intersect [a_{i_{h+1}},a_1) \subseteq I_{a_{i_{h+1}}} \intersect [a_{i_{h+1}},a_1)$. Therefore in order to show $J^h \in \M$, we need $|J^{h-1} \intersect (z,b_h]| \geq |I_{a_{h+1}} \intersect (z,b_h]|$ for all $z \in [a_1,b_h]$ from Lemma~\ref{lem:push}.

From lemma~\ref{lem:earlymax}, there exists a sequence $i_1 < \cdots < i_q < i_{q+1}=h$ (setting $t=h-1$ and $p=h$ in the lemma) such that $J^{h-1}$ maximizes $E' = [a_1,b_{i_1}], \ldots , [a_{i_q+1},b_{h}]$ in $\M$. Since $I_{a_1}$ and $J^{h-1}$ have the same number of elements in $[a_1,b_{i_1}]$ and $I_{a_1} \leq J^{h-1}$, we have $|J^{h-1} \intersect (z,b_{i_1}]| \geq |I_{a_1} \intersect (z,b_{i_1}]|$ for all $z \in [a_1,b_{i_1}]$.


From $J^1,\ldots,J^{h-1}$ having gaps, $J^{h-1}$ minimizes each $(b_i,a_{i+1})$ within $(b_s,b_h]$. From $J_2^1, \ldots, J_2^{h-1}$ having gaps, $J_2^{h-2}$ minimizes each $(b_i,a_{i+1})$ within $(b_1,b_h]$. This implies that within intervals of form $(b_i,a_{i+1})$ contained in $(b_1,b_h]$, the sets $J^{h-1}$ and $J_2^{h-2}$ are exactly the same. From the sharing property of morphs, $J^{h-1}$ maximizing $[a_1,b_{i_1}] , \ldots , [a_{i_q+1},b_{h}]$ implies $J_2^{h-2}$ also does too except potentially at $[a_1,b_{i_1}]$. Combining these facts, we get that $J^{h-1} \intersect (b_{i_1},b_h] = J_2^{h-2} \intersect (b_{i_1},b_h]$.


We now have all the ingredients to show that $|J^{h-1} \intersect (z,b_h]| \geq |I_{a_{h+1}} \intersect (z,b_h]|$ for each $z \in [a_1,b_h]$. From $J_2^{h-1} \in \M$ and having gaps, Lemma~\ref{lem:push} tells us that the same inequality replacing $J^{h-1}$ with $J_2^{h-2}$ is true for $z \in (b_1,b_h]$. Therefore for any $z \in (b_{i_1},b_h]$, we have $|J^{h-1} \intersect (z,b_h]| = |J_2^{h-2} \intersect (z,b_h]| \geq |I_{a_{h+1}} \intersect (z,b_h]|$. For any $z \in [a_1,b_{i_1}]$, we have $|J^{h-1} \intersect (z,b_h]| = |J^{h-1} \intersect (z,b_{i_1}]| + |J^{h-1} \intersect (b_{i_1},b_h]| \geq |I_{a_1} \intersect (z,b_{i_1}]| + |J_2^{h-2} \intersect (b_{i_1},b_h]| \geq |I_{a_{h+1}} \intersect (z,b_h]|$, since we have $I_{a_1} \intersect (z,b_{i_1}) \supseteq I_{a_{h+1}} \intersect (z,b_{i_1}]$ from the sharing property.



\end{proof}

The above proposition will be used as a key idea during the proof of the main result in the next section.

\section{Rank of arbitrary sets}

Let $E$ be a subset of the ground set as in Remark~\ref{rem:E}. We use $E_i$ to denote $[a_i,b_i]$. The rank of $E$ is bounded above by $\rk{\M}$ minus the sum of the minimal number of elements that a basis of $\M$ can possibly have in each cyclic interval of the complement of $E$. So we get $\rk{E} \leq \rk{\M} - \sum_i \minelts{b_i,a_{i+1}}$, where $\minelts{b,a}$ stands for the minimal number of elements that a basis of $\M$ can have in the interval $(b,a)$.  We call this bound the \newword{natural rank bound of $E$}: $\nbd{E} := \rk{\M} - \sum_i (\minelts{b_i,a_{i+1}})$. Notice that $\minelts{b,a} = \rk{\M} - \rk{[a,b]}$.

\begin{definition}
Let $\Pi$ be a partition $T_1 \sqcup \cdots \sqcup T_p$ of $[s]$ into pairwise disjoint non-empty subsets. We say that $\Pi$ is a \newword{non-crossing partition} if there are no cyclically ordered $a,b,c,d$ such that $a,c \in T_i$ and $b,d \in T_j$ for some $i \not = j$. We will call the $T_i$'s as the \newword{blocks} of the partition.
\end{definition}

To illustrate with a simple example, $\{1,3\}\sqcup\{2\}\sqcup\{4\}$ is a non-crossing partition of $[4]$, but $\{1,3\}\sqcup\{2,4\}$ is not. This can be easily verified by drawing the points $1$ to $4$ on a circle and trying to cut the circle into distinct regions corresponding to the partitions; this can only be done in the case of non-crossing partitions.

Let $\Pi$ be an arbitrary non-crossing partition of $[s]$ with $T_1,\ldots,T_p$ as its parts. We define $E|_{T_i}$ as the subset of $E$ obtained by taking only the intervals indexed by elements of $T_i$. For example, $E|_{\{1,3\}}$ would stand for $E_1 \cup E_3$. By submodularity of the rank function, we get another upper bound on the rank of $E$ : $\rk{E} \leq \rk{E|_{T_1}} + \cdots + \rk{E|_{T_p}} \leq \nbd{E, \Pi} := \nbd{E|_{T_1}} + \cdots + \nbd{E|_{T_p}}$. So for each non-crossing partition of $[s]$, we get an upper bound on the rank of $E$. We show that one of those bounds has to be tight in the theorem below.

\begin{theorem}
\label{thm:rk}
Let $E = [a_1,b_1] \union \cdots \union [a_s,b_s]$ be a disjoint union of $s$ cyclic intervals, where $a_1,b_1,a_2,b_2,\ldots,a_s,b_s$ are cyclically ordered. We have $\rk{E} = \nbd{E,\Pi}$ for some non-crossing partition $\Pi$ of $[s]$.
\end{theorem}
\begin{proof}
We use induction on $s$, the number of disjoint cyclic intervals of $E$. In case $s=1$, we have $\rk{E} = \rk{\M} - \minelts{E^c} = \nbd{E} = \nbd{E, \{\{1\}\} }$. Assume for the sake of induction that the claim is true for $1,\ldots,s-1$ intervals. We define $J_i^t$ recursively as in the previous section. From Proposition~\ref{prop:push}, we either have some $J_i^t \in \M$ that is gap-free or we have $J^{s-1} \in \M$ that isn't gap-free. In the latter case, since $J^{s-1}$ minimizes every interval of form $(b_i,a_{i+1})$, we have $|J^{s-1} \intersect E| = \nbd{E} = \nbd{E,\{\{1,\ldots,s\}\}}$.

Therefore we only have to take care of the case when we have some $J_i^t \in \M$ that is gap-free. Without loss of generality, we will assume $i=1$. From Lemma~\ref{lem:earlymax}, we have some sequence $i_1 < \cdots < i_q < i_{q+1}=t$ such that $J_i^t$ maximizes $[a_1,b_{i_1}], \ldots, [a_{i_q+1},b_{t}],[a_{t+1},b_s]$ (the last interval is maximized due to $J_i^t$ being gap-free). We will use $F_1,\ldots,F_{q+2}$ to denote these intervals. For each $1 \leq j \leq q+2$, let $K_j$ be a basis that maximizes $F_j \intersect E$. Modify $K_j$ using Lemma~\ref{lem:transf} so that it minimizes the complement of $F_j$ in $[n]$. Since $|K_j \intersect F_j| = |J^t \intersect F_j|$, using Proposition~\ref{prop:exch} we can replace $J^t \intersect F_j$ with $K_j \intersect F_j$ in $J^t$ for each $j$ to obtain a new basis $B$. By induction hypothesis, for each $j$, we have $|B \intersect F_j \intersect E| = \rk{F_j \intersect E} = \nbd{F_j \intersect E, \Pi_j}$ for some non-crossing partition $\Pi_j$. Letting $\Pi$ be a non-crossing partition obtained by collecting all blocks of $\Pi_j$'s, we get $|B \intersect E| = \rk{E} = \nbd{E,\Pi}$.

\end{proof}

\begin{figure}

\begin{center}
\includegraphics[scale = 0.7]{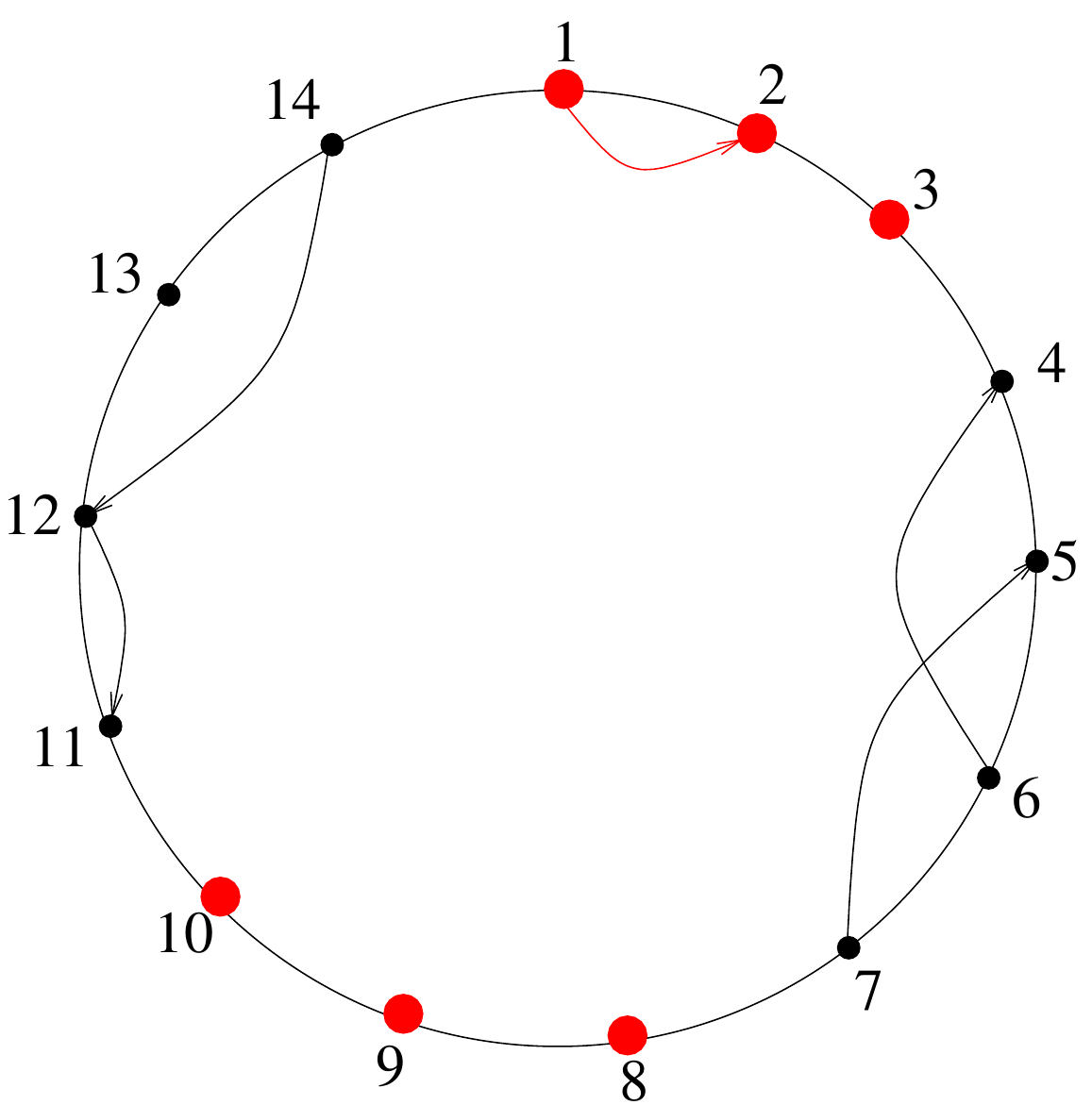}

    \caption{Information needed to compute the rank of $[1,3] \union [8,10]$.}
		\label{fig:interval13810rk}

\end{center}
\end{figure}

For example, take a look at Figure~\ref{fig:interval13810rk} (the positroid is the one associated to Figure~\ref{fig:permutation}). The rank of $E = [1,3] \union [8,10]$ is bounded above by $\nbd{E, \{\{1\},\{2\}\}}$ and $\nbd{E,\{\{1,2\}\}}$. We get $\nbd{E,\{\{1\},\{2\}\}} = \rk{[1,3]} + \rk{[8,10]} = 2 + 3 = 5$, since rank of an interval $[a,b]$ is given by $|[a,b]|$ minus the number of intervals of form $[\pi^{-1}(x),x]$ contained in $[a,b]$ (from $I_a$ being given by $a$-exceedances, and $\rk{[a,b]} = |I_a \intersect [a,b]|)$. We also have $\nbd{E,\{\{1,2\}\}} = \rk{\M} - \minelts{(3,8)} - \minelts{(10,1)} = 7 - 2 - 2 = 3$, since $\minelts{(b,a)}$ is given by the number of intervals of form $[x,\pi^{-1}(x)]$ contained in $(b,a)$. Hence the above theorem tells us that $\rk{E} = 3$.

\section{Application}

Let $\M$ be a positroid and let $E$ be an arbitrary subset of the ground set $[n]$. In this section, we will show how to use Theorem~\ref{thm:rk} to obtain the rank of $E$. We will call an interval of form $[x,\pi(x)]$ a \newword{CW-arrow}, and an interval of form $[x,\pi^{-1}(x)]$ a \newword{CCW-arrow} (each standing for clockwise and counterclockwise). Given a cyclic interval $T$, we use $\cw{T}$ to denote the number of CW-arrows contained in $T$. Similarly, we will use $\ccw{T}$ for the number of CCW-arrows contained in $T$. These numbers can easily be read from the associated decorated permutation of $\M$.

Recall that $\nbd{[a_1,b_1] \cup \cdots \cup [a_s,b_s]} = \rk{\M} - \sum_i (\minelts{b_i,a_{i+1}})$. And $\minelts{b_i,a_{i+1}}$ stands for the minimal possible number of elements a basis can have in the interval $(b_i,a_{i+1})$, which equals the number $|I_{a+1} \intersect (b_i,a_{i+1})|$. Hence $\minelts{b_i,a_{i+1}} = \ccw{(b_i,a_{i+1})}$. This gives us another way to interpret $\nbd{E}$ : it is $\rk{\M}$ minus the total number of CCW-arrows contained in the complement of $E$. In the special case when $E$ is a cyclic interval, $\nbd{E}$ is given by $|E|$ minus the number of CW-arrows contained in $E$.

Therefore for any $E$, we can obtain $\nbd{E,\Pi}$ by counting CW-arrows and CCW-arrows. If $E$ is the disjoint union of $s$ cyclic intervals, we first write all possible non-crossing partitions of $[s]$. Each one of them gives a sum of $\nbd{E'}$'s where $E'$ obtained from $E$ by taking some of the $s$ cyclic intervals of $E$, and we compute them by counting the CCW-arrows (or CW-arrows for intervals) of the decorated permutation.

Consider the positroid associated with Figure~\ref{fig:permutation}. Let us try to compute the rank for $E = [1,2] \union [7,10] \union [13,13]$. We have $3$ disjoint intervals, so the upper bounds of $\rk{E}$ will be coming from the non-crossing partitions of $\{1,2,3\}$. The following are the upper bounds for $\rk{E}$ we get:
\begin{itemize}
\item $\nbd{E,\{\{1,2,3\}\}} = \nbd{E} =  \rk{\M} - \ccw{(2,7)} - \ccw{(10,13)} - \ccw{(13,13)} = 7 - 1 - 1 - 0 = 5.$
\item $\nbd{E,\{\{1\},\{2,3\}\}} = \nbd{E_1} + \nbd{E_2 \union E_3} = |[1,2]| - \cw{[1,2]} + \rk{\M} - \ccw{(10,13)} - \ccw{(13,7)} = 1 + 7 - 1 - 1 = 6.$
\item $\nbd{E,\{\{1,2\},\{3\}\}} = \nbd{E_1 \union E_2} + \nbd{E_3} = \rk{\M} - \ccw{(2,7)} - \ccw{(10,1)} + |[13,13]| - \cw{(13,13)} = 7 - 1 - 2 + 1 - 0 = 5.$
\item $\nbd{E,\{\{1,3\},\{2\}\}} = \nbd{E_1 \union E_3} + \nbd{E_2} = \rk{\M} - \ccw{(2,13)} - \ccw{(13,1)} + |[7,10]| - \cw{[7,10]} = 7 - 5 - 0 + 4 - 0 = 6.$
\item $\nbd{E,\{\{1\},\{2\},\{3\}\}} = \nbd{E_1} + \nbd{E_2} + \nbd{E_3} = |[1,2]| - \cw{[1,2]} + |[7,10]| - \cw{[7,10]} + |[13,13]| - \cw{[13,13]} = 2 - 1 + 4 - 0 + 1 - 0  = 6.$
\end{itemize}

Theorem~\ref{thm:rk} tells us that $\rk{E} = 5$.

Using Theorem~\ref{thm:rk}, in \cite{OX} it is shown that the facets of the matroid polytope are given by cyclic intervals whose complement is covered by CCW-arrows. It is also shown that the facets of the independendent set polytope of a positroid are given by sets whose complement is again covered by CCW-arrows. In \cite{CLOZ}, the condition for an arbitrary subset of the ground set being a flat of the positroid will be given in terms of the decorated permutation, again using Theorem~\ref{thm:rk}.

\bibliographystyle{plain}    
\bibliography{bibliography}

\end{document}